\documentclass[reqno,11pt,graphicx,]{amsart}
\usepackage[colorlinks,linkcolor=blue,anchorcolor=yellow,citecolor=red,urlcolor=green]{hyperref} 
\usepackage{amsmath,amscd,amsthm,amsfonts,appendix,tikz}
\usepackage{amssymb,mathrsfs,xcolor,bbm,dsfont,verbatim}
\usepackage{enumitem}
\usepackage[OT2,OT1]{fontenc}
\usepackage{caption}
\usepackage{longtable}
\usepackage[all]{xy}
\numberwithin{equation}{section}
\allowdisplaybreaks[4]
\theoremstyle{plain}
\newcommand{\N}{{\mathbb N}}

\newcommand{\F}{{\mathcal F}}
\newcommand{\R}{{\mathbb R}}

\newcommand{\M}{{\mathcal M}}

\newtheorem{theorem}{Theorem}[section]
\newtheorem{proposition}{Proposition}[section]
\newtheorem{lemma}{Lemma}[section]
\newtheorem{remark}{Remark}[section]
\newtheorem{definition}{Definition}[section]

\allowdisplaybreaks

\begin{document}
	\title{Equilibrium States, Zero Temperature Limits and Entropy Continuity for Almost-Additive Potentials}
	
	\author{Jie CAO}
	\address[J. Cao]{Chern Institute of Mathematics and LPMC, Nankai University, Tianjin 300071, P. R. China.}
	\email{caojie@nankai.edu.cn}
	\maketitle	
	\begin{abstract}
		This paper is devoted to study the equilibrium states for almost-additive potentials defined over topologically mixing countable Markov shifts (that is a non-compact space) without the big images and preimages (BIP) property. Let $\F$ be an almost-additive and summable potential with bounded variation potential. We prove that there exists a unique equilibrium state $\mu_{t\F}$ for each $t>1$ and there exists an accumulation point $\mu_{\infty}$ for the family $(\mu_{t\F})_{t>1}$ as $t\to\infty$. We also obtain that the Gurevich pressure $P_{G}(t\F)$ is $C^1$ on $(1,\infty)$ and the Kolmogorov-Sinai entropy $h(\mu_{t\F})$ is continuous at $(1,\infty)$. As two consequences, we extend completely the results for the zero temperature limit [J. Stat. Phys. ,155 (2014),pp. 23–46] and entropy continuity at infinity [J. Stat. Phys., 126 (2007),pp. 315–324] beyond the finitely primitive case. We also extend the result [Trans. Amer. Math. Soc., 370 (2018), pp. 8451–8465] for almost-additive potentials. 
	\end{abstract}
	
	\section{Introduction}
	
	The thermodynamic formalism is a branch of the ergodic theory that studies
	existence, uniqueness and properties of equilibrium states. In this paper we study the equilibrium state associated with almost-additive potential $\F=\{f_n\}_{n=1}^{\infty}$ (see Definition \ref{def-almost-additive}) defined over countable Markov shifts (CMS), that is, measures that maximize the value 
	\begin{equation*}
		h(\mu) + \F_*(\mu),
	\end{equation*}
	where $h(\mu)$ is the measure entropy (is also called {\it Kolmogorov-Sinai entropy}) (see \cite[Chapter 4]{Wa1982}), and $\F_*(\mu):=\lim_{n\to\infty}\frac{1}{n}\int f_n\, d\mu$ is the {\it Lyapunov exponent} for $\mu$, provided that the limit exists. Assuming uniqueness of equilibrium states 
	$\mu_{t\F}$ for each scaled almost-additive potential $t\F$ with $t > 1$, we address three fundamental problems: the existence of accumulation points for the family $(\mu_{t\F})_{t>1}$, the zero temperature limit and entropy continuity at infinity. 
	
	If the potential $\F$ is additive (that is $f_n=S_n f$ is the Birkhoff sum of $f$). A fundamental result in thermodynamic formalism \cite{BS2003} establishes that equilibrium states exhibit at most one equilibrium state, with existence typically requiring stringent dynamical conditions. Traditional approaches often assume the incidence matrix satisfies strong regularity properties--notably finite primitiveness or, in topologically mixing shifts with BIP condition \cite{Sa2003}, which are equivalent to the existence of Gibbs states. These classical criteria effectively constrain the system's combinatorial complexity to ensure analytical tractability. In a significant advancement, Freire and Vargas \cite{FV2018} introduced a novel criterion that relaxes these prerequisites. By circumventing the finite primitivity assumption, their work extends equilibrium state existence theory to broader classes of symbolic dynamics, building upon earlier developments in \cite{MU2001,Sa1999}. 
	
	The thermodynamic formalism for almost-additive potentials over
	CMS was developed in \cite{IY2012}. In particular, the existence of Gibbs state was established in \cite[Theorem 4.1]{IY2012}. Later in \cite[Proposition 3.1]{IY2014}, Iommi and Yayama introduced suitable conditions that ensure that certain Gibbs state is actually equilibrium state (this is not always the case in non-compact setting, see \cite[p. 1757]{Sa2003}).
	
	Our main result for the equilibrium state is as follows:.
	\begin{theorem}\label{equilibrium state}
		Let $(\Sigma,\sigma)$ be a topologically mixing CMS and the potential $\F$ be almost-additive and summable with bounded variation. Then, for any $t > 1$ there is a unique equilibrium state $\mu_{t\F}$ associated to $t\F$. Also, there exists an accumulation point $\mu_{\infty}$ for the family $(\mu_{t\F})_{t>1}$ as $t\to\infty$.
	\end{theorem}
	As previously established, our methodological advancement resides in circumventing the necessity of Gibbs state across the entire space (the space without BIP property). Instead, we employ the condition that the almost-additive potential $\F$ exhibits summability with bounded variation, thereby ensuring the existence of equilibrium states $\mu_{t\F}$ for each parameter $t > 1$. It is essential to emphasize that while we demonstrate the existence of accumulation points for the sequence $(\mu_{t\F})_{t>1}$, this does not inherently guarantee the existence of the limit $\lim_{t \to \infty}\mu_{t\F}$. To illustrate this subtlety, consider the simpler setting of compact subshifts of finite type with H\"older continuous potentials: Hochman and Chazottes \cite{CH2010} constructed an example where such convergence fails. This counterexample underscores the inherent complexity of equilibrium state behavior even in well-structured dynamical systems. However, under certain finite range assumptions convergence has been proved in \cite{BLMV2023,Br2003,CGU2011,Le2005}.
	
	The other major result establishes that when the almost-additive potential $\F$ is summable with bounded variation, its Gurevich pressure maintains $C^1$-property, equilibrium entropy and Lyapunov exponent are continuous: 
	\begin{theorem}\label{relatived-pressure-diff}
		Let $\mu_{t\F}$ be the equilibrium state for $t\F\,(t>1)$ and write 
		\begin{equation*}
			P(t):=P_{G}(t\F);\ \ \ L(t):=\F_*(\mu_{t\F});\ \ \ 	H(t):=h(\mu_{t\F}),
		\end{equation*}
		where $P_G(\F)$ is the {\it Gurevich pressure} of $\F$ (see Definition \ref{def-pre}).
		Then
		
		(1) $P$ is $C^1$ and strictly convex on $(1,\infty)$. Moreover,
		$
		P'(t)=L(t).
		$
		
		(2) $-L,\,H$ are continuous and strictly decreasing on $(1,\infty)$. Moreover, we have the weak limit $\lim_{t\to s}\mu_{t\F}=\mu_{s\F}$ for any $s\in(1,\infty)$.
	\end{theorem} 
	The regularity properties of the pressure are related to phase transitions. A cornerstone result by Ruelle \cite{Ru1968,Ru1978} establishes that for any topologically mixing finite Markov shift, the topological pressure function exhibits real-analytic behavior over the space of H\"older continuous potentials. Within the context of ferromagnetism, this analytical regularity has been interpreted as indicative of ``absence of phase transitions" (see \cite{El1985}). However, this elegant theoretical framework encounters fundamental limitations when extended to systems with countably states. There are examples include the Manneville–Pomeau map (see e.g. \cite{PM1980,Lo1993}) and the Farey map \cite{PS1992} (see also \cite{LSV1999}). Sarig \cite{Sa2003} studied the analyticity of the topological pressure for some one-parameter families of potentials on countable Markov shifts (see also \cite{MU2001,SU2003}).
	
	In statistical mechanics an important problem is that of describing how do equilibrium states varies as the temperature changes. Of particular importance is the case when the temperature decreases to zero (zero temperature limit). Indeed, this case is related to ground states, that is, measures supported on configurations of minimal energy \cite[Appendix B.2]{EFS1993}. A similar problem, given any observable $f:\Sigma\to\R$ we say that a $\sigma$-invariant measure $\mu$ is a {\it maximising} measure for $f$ if
	\begin{equation*}
		\int f\, d\mu=\sup\left\{\int f \, d\nu:\nu\in\M\right\},
	\end{equation*}
	where $\M$ denotes the set of $\sigma$-invariant probability measures on $\Sigma$. In certain cases, some maximising measures can be described as the limit of equilibrium states as the temperature goes to zero. The theory that studies maximising measures is usually called {\it ergodic optimisation} (see \cite{Bo2000,Je2019} for more details).
	
	We also study a similar problem for the almost-additive potentials and obtain that the continuity of Lyapunov exponent $L(t)$ at $t\to\infty$:
	\begin{theorem}\label{entropy-continuous}
		Under the hypotheses of Theorem \ref{equilibrium state}. Then
		\begin{equation*}
			\lim_{t\to\infty}\F_*(\mu_{t\F})=\F_*(\mu_{\infty})=\sup\left\{\F_*(\nu):\nu\in\M\right\}.
		\end{equation*}
		As a consequence, $\mu_{\infty}$ is a maximising measure for $\F$.
	\end{theorem}
	The above theorem extends the results of Iommi and Yayama \cite{IY2014}, while employing a distinct proof methodology. Furthermore, we emphasize that our conclusions do not require the space $(\Sigma,\sigma)$ to satisfy the BIP property, and the conditions imposed on the potential $\F$ are significantly more general. A similar problem, in the case of additive potential, was first studied by Jenkinson, Mauldin and Urban\'ski \cite{JMU2005} (see also \cite{BG2010,BF2014,Io2007,Ke2011}).
	
	For the continuity of equilibrium entropy at infinity, under some appropriate conditions, Morris \cite{Mo2007}, Freire and Vargas \cite{FV2018} proved the existence of the limit $\lim_{t \to \infty} h(\mu_{t\F})$ with the additive potential $\F=\{S_nf\}_{n=1}^{\infty}$. Also, they showed that this limit agree with the supremum of the entropies over the set of the maximizing measures for $f$. 
	
	We are able to generalize this result to almost-additive potentials:
	\begin{theorem}\label{zero-continuous}
		Under the hypotheses of Theorem \ref{equilibrium state}. Then
		\begin{equation*}
			\lim_{t \to \infty}h(\mu_{t\F})=h(\mu_{\infty}) = \sup\{h(\mu):\mu \in \mathcal{M}_{max}(\F)\},
		\end{equation*}
		where $\mathcal{M}_{\max}(\F)$ denotes the set of maximizing measures. 
	\end{theorem}
	\begin{remark}
		This answers two questions posed by Jenkinson,Mauldin and Urba\'nski \cite{JMU2005} for almost-additive potentials.
	\end{remark}
	
	The paper is organized as follows. In Sect. \ref{Sec-preliminaries}, we give the basic definitions and notations that we use in the proofs and results. In Sect. \ref{Sec-compact-app}, we construct the compact subshifts approximation to prove our results. In Sect. \ref{Sec-proof-Theorem-1}, we use the existence of the equilibrium states on the compact case to show the existence of a unique equilibrium state in the non-compact case, therefore proving Theorem \ref{equilibrium state}. In Sect. \ref{Sec-proof-Theorem-234}, we establish Theorem \ref{relatived-pressure-diff} by leveraging the uniqueness of equilibrium states, and subsequently demonstrate Theorems \ref{entropy-continuous} and \ref{zero-continuous} as direct applications of this foundational result.
	
	\section{Preliminaries}\label{Sec-preliminaries}
	Let $\mathcal A$ be a countable alphabet set and $\mathbf{M}=(M_{i j})_{\mathcal A\times \mathcal A}$ be a matrix of zeros and ones (with no row and no column made entirely of zeros). The symbolic space associated to $\mathbf{M}$
	with alphabet $\mathcal{A}$ is defined by
	\begin{equation*}
		\Sigma:=\left\{x=x_1x_2\cdots\in \mathcal A^\N: M_{x_ix_{i+1}}=1,\ \forall\, i\in \N\right\}.
	\end{equation*}
	The {\it shift map} $\sigma:\Sigma\to\Sigma$ is defined by $(\sigma(x))_i:=x_{i+1}$. Note that $\sigma$ is a continuous map. The pair $(\Sigma,\sigma)$ is called a {\it countable Markov shift} (CMS). 
	
	Define a metric $d$ on $\Sigma$ as
	$
	d(x,y):=\exp(-\min\{i\in\N:\,x_i\neq y_i\}).
	$
	Then $(\Sigma,d)$ become a metric space, but not  compact since $\mathcal A$ is infinite. Just to ease the calculations we can suppose that $\mathcal{A}=\N$. 
	
	Let $w=w_1w_2\cdots w_n\in \mathcal A^n$, we say that $w$ is {\it admissible} of length $n$, if $M_{w_iw_{i+1}}=1$ for any $1\le i<n$. Write $a\xrightarrow{n} b$ if there is an admissible word of length $n+1$ which starts at $a$ and ends at $b$. Let 
	$|w|$ denote the length of an admissible word $w$.
	
	For any admissible word $w=w_1w_2\cdots w_n$ of length $n$, define the {\it cylinder} as
	\begin{equation*}
		[w]:=\left\{x\in \Sigma: x|_n=w\right\} \ \ \text{ where } \ \ x|_n:=x_1\cdots x_n.
	\end{equation*}
	We equip $\Sigma$ with the topology generated by the cylinders sets. 
	
	Assume that $(\Sigma,\sigma)$ is a CMS. We say that $(\Sigma,\sigma)$ is {\it topologically mixing}, if for any $a,b\in \mathcal A$, there exists $N_{ab}\in \N$ such that for any $n\ge N_{ab}$, there exists an admissible word $w=w_1w_2\cdots w_n$ such that $w_1=a$ and $w_n=b$, that is, $a\xrightarrow{n-1}b$.

	\begin{definition}\label{def-almost-additive}
		Let $(\Sigma,\sigma)$ be a CMS. For each $n\in\N$,
		let $f_n:\Sigma\to\R$ be a continuous function. The potential $\mathcal{F}=\{f_n\}_{n=1}^{\infty}$
		is called almost-additive if there exists a constant $C=C_{aa}(\F)\geq0$ such that for every $n,m\in\N,x\in\Sigma$, 
		\begin{equation*}
			f_n(x)+f_m(\sigma^nx)-C\leq f_{n+m}(x)\leq f_n(x)+f_m(\sigma^nx)+C.
		\end{equation*}
		In particular, if $C=0$, then $f_n=S_nf_1=\sum_{i=1}^{n-1}f_1\circ\sigma^{n-1}$ and $\F$ is called additive. 
	\end{definition}
	\begin{definition}
		Let $(\Sigma,\sigma)$ be a CMS. The potential $\mathcal{F}=\{f_n\}_{n=1}^{\infty}$ has bounded variation (is also called a Bowen sequence), if there exists a constant $C=C_{bv}(\F)\geq0$ such that
		\begin{equation*}
			\sup_{n}\left\{|f_n(x)-f_n(y)|:x,y\in\Sigma,\, x_i=y_i,\, 1\leq i\leq n\right\}\leq C.
		\end{equation*}
	\end{definition}
	It is noteworthy that potentials endowed with the almost-additive and bounded variation properties constitute a generalization of summable variation functions. Indeed, for a continuous function $f:\Sigma\to\R,$ let $var_n(f)=\{|f (x)-f(y)|: x,y\in\Sigma,\, x_i = y_i,\,  1\leq i\leq n\}$. A function $f$ is of {\it summable variation} if $\sum_{i=1}^{\infty}var_i(f)<\infty$. This regularity assumption has been widely used in ergodic theory \cite{Ba2011,Bo2008,ITV2022,Wa1982}
	
	The following pressure definition extends Sarig's \cite{Sa1999} to almost-additive potentials:
	\begin{definition}\label{def-pre}
		Assume $(\Sigma,\sigma)$ is a topologically mixing CMS. Let $\mathcal{F}=\{f_n\}_{n=1}^{\infty}$ be an almost-additive with bounded variation potential, the {\it Gurevich pressure} of $\mathcal{F}$ is defined as
		\begin{equation*}\label{def-pressure}
			P_G(\mathcal{F}):=\lim_{n\to \infty}\frac{1}{n} \log\Big(\sum_{\sigma^n x=x}\exp(f_n(x))\chi_{[a]}(x)\Big),\ \ \ a\in\mathcal{A}.
		\end{equation*}
	\end{definition}
	\begin{remark}
		Since $(\Sigma,\sigma)$ is topologically mixing, the Gurevich pressure is independent from the choice of $a\in \mathcal{A}$, see also \cite[Theorem 2.1 (3)]{IY2012}. 
	\end{remark}
	The following proposition is a generalisation of the variational principle for continuous
	functions to the setting of almost-additive potential (see \cite[Proposition 3.1 and Theorem 3.1]{IY2012}). In order to state it we need the following definition, given $f:\Sigma\to\R$
	a continuous function, the {\it Ruelle operator} $L_{f}$ applied to function $g:\Sigma\to\R$ is formally
	defined by
	\begin{equation*}
		\left(L_{f}g\right)(x):=\sum_{\sigma y=x}e^{f(y)}g(y),\quad\text{for every }x\in\Sigma.
	\end{equation*}
	
	\begin{proposition}[\cite{IY2012}]\label{IY-variational principle}
		Let $(\Sigma,\sigma)$ be a topologically mixing CMS. Assume $\F=\{f_n\}_{n=1}^{\infty}$ is almost-additive with bounded variation property such that $\|L_{f_1}1\|_{\infty}<\infty$. Then $-\infty<P_{G}(\F)<\infty$ and
		\begin{align*}
			P_{G}(\F)&=\sup\left\{h(\mu)+\F_*(\mu):\mu\in\mathcal{M}\ \text{and}\ \F_*(\mu)>-\infty\right\}\\
			&= \sup \{ P_{G}(\F|_{\Sigma'}) :\Sigma'\subset\Sigma\ \text{a topologically mixing finite state Markov shift}\,\}.\label{VPCompact}
		\end{align*}
		In particular, if there exists $\mu\in\M$ such that
		$
		P_{G}(\F)=h(\mu)+\F_*(\mu),
		$
		then $\mu$ is called an equilibrium state for potential $\F$.
	\end{proposition}
	\begin{remark}
		Under the hypotheses of Proposition \ref{IY-variational principle}, one can check that pressure function $t\mapsto P_{G}(t\F)$ is convex on $(1,\infty)$, see also \cite[Corollary 3.2]{IY2012}.
	\end{remark}
	\begin{definition}
		Let $(\Sigma,\sigma)$ be a CMS. The potential $\mathcal{F}=\{f_n\}_{n=1}^{\infty}$ is called summable if 
		\begin{equation}\label{eqsum}
			\sum_{i \in \mathbb{N}} \exp(\sup (f_1|_{[i]})) < \infty \,.
		\end{equation}
	\end{definition}
	In this case, we observe that $\sup f_1,\,\|L_{f_1}1\|_{\infty}<\infty$, but $\lim_{i\to\infty}\sup f_1=-\infty$. Moreover, it is shown in \cite[Lemma 3.1]{Mo2007}\footnote{One can easily notice the BIP property is not required for this.} that the summability condition implies, for any $t > 1$, that
	\begin{equation}\label{eqsumt}
		\sum_{i \in \mathbb{N}} \sup (-tf_1|_{[i]})\exp(\sup (tf_1|_{[i]})) < \infty \,.
	\end{equation}
	The similar condition is considered in \cite{IY2014}. It is the key ingredient to prove in this paper the  existence of the equilibrium state associated to $t\F$ for each $t > 1$.
	
	\begin{definition}
		Let $(\Sigma,\sigma)$ be a topologically mixing CMS and $\F=\{f_n\}_{n=1}^{\infty}$ be an almost-additive potential on $\Sigma$. A measure $\mu\in\M$ is said to be a Gibbs state for $\F$ if there is a constant $C_2>C_1>0$ such that for any $x \in \Sigma$ and each $n \geq 1$,
		\begin{equation}\label{gibbs}
			C_1\leq \frac{\mu[x_1 \ldots x_{n}]}{\exp( f_n(x) - nP_{G}(\F))} \leq C_2 \,.
		\end{equation}
	\end{definition}

	When $\Sigma$ is compact, the equilibrium states and the Gibbs states are unique and agree, see for example, \cite[Lemma 18]{Mu2006}. Furthermore, there must exist a Gibbs state for almost-additive potential with bounded variation, and we can choose (see Appendix \ref{appendix})
	\begin{equation}\label{C2}
		C_2=\exp(C_{bv}(\F)). 
	\end{equation}
	\section{Compact subshifts approximation}\label{Sec-compact-app}
	
	Throughout this paper, we assume that the symbolic space $(\Sigma,\sigma)$ is topologically mixing and potential $\F$ is almost additive and summable with bounded variation. 
	Our aim is to prove the existence of equilibrium states for $t\F$ and accumulation points for such sequences. The techniques we use are inspired from \cite[Sect. 3]{FV2018}, and we employ a strategic approximation framework. Specifically, we construct sequences of compact subshifts of $\Sigma$ and analyze the corresponding Gibbs equilibrium states restricted to these subsystems.
	
	At first, we have the following compact subshifts approximation lemma:
	\begin{lemma}\label{approximation-lemma}
		There exists a sequence \((\Sigma_k)_{k \in \mathbb{N}}\) of compact topologically mixing subshifts of \(\Sigma\) such that \(\Sigma_k \subset \Sigma_{k+1}\) and
		\begin{equation}\label{VPApproximation}
			P(t) = \sup \big\lbrace P_k(t) : k \in \mathbb{N} \big\rbrace = \lim_{k \to \infty} P_k(t),\quad\text{for all}\ t>1,
		\end{equation}
		where $P_k(t) := P_G(t\mathcal{F}|_{k})$ and $\mathcal{F}|_{k} = \lbrace f_{n;k} \rbrace_{n=1}^{\infty}$ denotes the restriction of \(\mathcal{F}\) to $\Sigma_k$.
	\end{lemma}
	\begin{proof}
We define the compact topologically mixing subshift $\Sigma_k\subset\Sigma$ inductively as follows:
		
		At first, let $\mathcal{S}_1=\{1\}$. There exist $N_1\geq2$ and sets $
		\{e_1,\cdots, e_{N_1-1}\},\, \{c_1,\cdots, c_{N_1}\}\subset\N
		$ such that $1\, e_1e_2\cdots e_{N_1-1}1$ and $1\, c_1c_2\cdots c_{N_1}1$ are admissible words by topologically mixing property. Then we define 
		\begin{equation*}
			\Sigma_1:=\left\{x=x_1x_2\cdots\in\mathcal{A}_1^{\N}:M_{x_nx_{n+1}}=1,\ \forall\,n\in\N\right\},
		\end{equation*}
where $\mathcal{A}_1:=\mathcal{S}_1\cup\mathcal{B}_1$ and 
$\mathcal{B}_1=\left\{e_i,c_j:1\leq i\leq N_1-1;\,1\leq j\leq N_1\right\}$.

		Next, assume that for any $k\geq2$, the subshifts $\Sigma_{k-1}$ have been defined. Write $\mathcal{S}_k=\mathcal{S}_{k-1}\cup\mathcal{B}_{k-1}$ and take $N_k=\sup\{N_{ab}:a,b\in\mathcal{S}_k\}\in\N$ is such that for each $i,j\in\mathcal{S}_k$, there exist sets $
		\{e_1^{ij},e_2^{ij},\cdots, e_{N_k-1}^{ij}\},\,\{c_1^{ij},c_2^{ij},\cdots, c_{N_k}^{ij}\}\subset\N
		$ such that the words $i\, e_1^{ij}e_2^{ij}\cdots e_{N_k-1}^{ij}\,j$ and $i\, c_1^{ij}c_2^{ij}\cdots c_{N_k}^{ij}\,j$ are admissible in $\Sigma$. Now we define $
		\mathcal{B}_k=\bigcup_{i,j\in\mathcal{S}_k}\{e^{ij}_p,c^{ij}_q:1\leq p\leq N_k-1;\,1\leq q\leq N_k\}
		$
		and $\mathcal{A}_k:=\mathcal{S}_k\cup\mathcal{B}_k$. Then we define 
		\begin{equation*}
			\Sigma_k:=\left\{x=x_1x_2\cdots\in\mathcal{A}_k^{\N}:M_{x_nx_{n+1}}=1,\ \forall\,n\in\N\right\}.
		\end{equation*}
		It is seen that $\mathcal{A}_k$ is finite and $\Sigma_k\subset\Sigma_{k+1}\subset\Sigma$ for all $k\in\N$.
		
		Now we prove that $\Sigma_k$ is topologically mixing for all $k\in\N$. For any $i,j\in\mathcal{A}_k$, there exist $a,b\in\mathcal{A}_k$ and $0\leq m_1=m_1(a,b),m_2=m_2(a,b)\leq N_k$ such that
	$		i \xrightarrow{m_1} a\ \ \text{and} \ \ b \xrightarrow{m_2} j$
		on $\Sigma_k$. 
		By the definitions of $\mathcal{S}_k$ and $\mathcal{B}_k$, for any $p,q\geq0$, we have
		\begin{equation}\label{admissible-con}
			i \xrightarrow{m_1}\underset{\substack{p}}{\underbrace{a \xrightarrow{N_{k}}a\cdots a\xrightarrow{N_{k}}a}} \xrightarrow{N_{k}} b\underset{\substack{q}}{\underbrace{ \xrightarrow{N_{k}+1}b\cdots b\xrightarrow{N_{k}+1}b}}\xrightarrow{m_2} j
		\end{equation}
		on $\Sigma_k$. Combine with \eqref{admissible-con}, we see that for all $i,j\in\mathcal{A}_k$, there exists $$N(i,j):=m_1+m_2+N_{k}+N_{k}(N_{k}+1)\in\N$$ such that for all $n\geq N(i,j)$, we have $i\xrightarrow{n} j$ by writing
		\begin{equation*}
			n=m_1+m_2+N_{k}+pN_{k}+q(N_{k}+1).
		\end{equation*}
		
		Since $\Sigma_k$ is topologically mixing, for each $k\in\mathbb{N}$ and $a\in\mathcal{A}_k$, we have 
		$$
		P_k(t)=P_G(t\F|_{k})=\lim_{n\to\infty}\frac{1}{n}\log\sum_{\sigma^n x = x}\exp(t f_n(x))\chi_{[a]}(x)\chi_{\Sigma_k}(x).
		$$
		Observe that $\Sigma_k\subset\Sigma_{k+1}\subset\Sigma$, then for any $t>1$ we have $P_k(t) \leq P_{k+1}(t) \leq P(t)$. Now equation \eqref{VPApproximation} just follows from Proposition \ref{IY-variational principle}.
	\end{proof}
	
	Let $\mu_{t,k}:=\mu_{t\F|_{k}}$ be the Gibbs state related to potential $t \F|_{k}$. Bellow we show that the sequence $(\mu_{t,k})_{k \in \mathbb{N}}$ in $\mathcal{M}$ has a convergent subsequence. For this, we  prove that this sequence is tight. Recall that a subset $\mathcal{K} \subset \mathcal{M}$ is tight if for every $\varepsilon > 0$ there is a compact set $K \subset \Sigma$ such that $\mu(K^c) <\varepsilon$ for any $\mu \in \mathcal{K}$.
	
	\begin{lemma}\label{lematight} For each $t > 1$, the equilibrium states sequence $(\mu_{t,k})_{k \in \mathbb{N}}$ is tight.
	\end{lemma}
	
	\begin{proof} 
		Our proof is similar to the proof in \cite{JMU2005,FV2018}. Assume that $\varepsilon> 0$ and
		$$
		K = \{x \in \Sigma : 1 \leq x_m \leq n_{m} \text{ for each $m \in \mathbb{N}$}\} \,,
		$$
		where $(n_m)_{m \in \mathbb{N}}$ in $\mathbb{N}$ is an increasing sequence (we choose large $n_m$ satisfies the inequality \eqref{def-n-m}).
		It is seen that the set $K$ is compact in the product topology and satisfies
		\begin{align}\label{eqcom1}
			\mu_{t,k}(K^c)
			&=\mu_{t,k}\Big(\bigcup_{m\in\N}\{x\in\Sigma:x_m>n_m\}\Big)\notag\\
			&\leq\sum_{m \in \mathbb{N}} \sum_{i > n_m} \mu_{t,k}(\{x\in\Sigma:x_m=i\})\leq\sum_{m \in \mathbb{N}} \sum_{i > n_m} \mu_{t,k}([i]).
		\end{align}
		
		Let $\mu \in \mathcal{M}$ such that $S=\F_*(\mu)$ satisfies $-\infty < S < \infty$ (it is sufficient to choose $\mu = \frac{1}{p}\sum^{p - 1}_{j = 0} \delta_{\sigma^j \bar{x}}$ with $\bar{x} \in Per_p(\Sigma_1)$). Notice that  $\Sigma_1\subset \Sigma_k$ for all $k \in \mathbb{N}$, we can consider $\mu$ to be a measure well defined both in $\Sigma$ or in $\Sigma_k$ for any $k$. Let $S_k =(\F|_{k})_*(\mu)$ and by the previous comment, we have that $S_k$ is well defined, and it is also clear that $S_k = S$ for any $k \in \mathbb{N}$. Furthermore, we have that
		\begin{equation}\label{eqcom2}
			P_k(t) - tS_k\geq h(\mu) + t(\F|_{k})_*(\mu)-t S_k = h(\mu) \geq 0, \,\ \text{for all}\ k\in\N.
		\end{equation}
		
		Since $\Sigma_k$ is compact, by \eqref{gibbs}, \eqref{C2} and the fact that $\exp(tC_{bv}(\F|_{k})) \leq \exp(tC_{bv}(\F))$, then for all $x \in [i]$ we have that
		\begin{equation}\label{compact-gibbs-measure}
			\frac{\mu_{t,k}[i]}{\exp(t f_{1;k}(x)-P_k(t))} \leq \exp(tC_{bv}(\F)), \,\ \text{for all}\ k\in\N.
		\end{equation}
		Recall that $S_k=S$ for all $k$, and by \eqref{compact-gibbs-measure} and  \eqref{eqcom2}, we obtain that for each $x \in [i]$,
		\begin{align*}
			\mu_{t,k}([i]) &\leq \exp(tC_{bv}(\F)+t f_{1;k}(x)-P_k(t)) \\
			&\leq\exp(t(C_{bv}(\F) + \sup f_1 |_{[i]} - S_k)) \exp(tS_k- P_k(t)) \\
			&\leq \exp(t(C_{bv}(\F)+\sup f_1 |_{[i]} - S)).               
		\end{align*}
		Since $\F$ is summable, then for $i$ large enough, we have $C_{bv}(\F) + \sup f_1 |_{[i]} - S \leq 0$ and 
		\begin{equation}\label{mu-upper}
			\mu_{t,k}([i])\leq \exp(C_{bv}(\F) + \sup f_1 |_{[i]} - S).
		\end{equation}
		
		Moreover, the summability condition (see \eqref{eqsum}) shows that one can choose the sequence $(n_m)_{m \in \mathbb{N}}$ satisfies
		\begin{equation}\label{def-n-m}
			\sum_{i > n_m} \exp(\sup f_1|_{[i]}) < \frac{\varepsilon}{2^{m + 1}} \exp(S -C_{bv}(\F)) \,.
		\end{equation}
		Therefore, by \eqref{eqcom1}, we deduce that
		$$
		\mu_{t,k}(K^c)\leq\sum_{m \in \mathbb{N}} \sum_{i > n_m} \mu_{t,k}([i])\leq\sum_{m\in\N} \exp(C_{bv}(\F)+\sup f_1|_{[i]} - S)\leq\sum_{m \in \mathbb{N}} \frac{\varepsilon}{2^{m + 1}} = \varepsilon \,.
		$$
		Now the proof is complete.
	\end{proof}
	
	\section{Proof of Theorem \ref{equilibrium state}}\label{Sec-proof-Theorem-1}
	
	In this section we prove our first theorem, which is a consequence of the tightness of the sequence $(\mu_{t,k})_{k \in \mathbb{N}}$ and two continuous lemmas, Lemma \ref{lema_lim_ut} and Lemma \ref{entropy-continuous-lemma}. 
	
	We can suppose, w.l.o.g. that $f_n\leq0$, since $\sup f_1$ is bounded above by the fact that it is almost-additive, so we can consider $f_n-n\sup f_1-nC_{aa}(\F)$ instead of $f_n$. 
	
	By Lemma \ref{lematight} and \emph{Prohorov's theorem}, for each $t>1$, there exists a subsequence $(\mu_{t,k_m})_{m \in \mathbb{N}}$ of the sequence $(\mu_{t,k})_{k \in \mathbb{N}}$ and a measure $\mu_t \in \mathcal{M}$ such that 
	\begin{equation}\label{weak-star}
		\mu_t = \lim_{m \to \infty} \mu_{t,k_m}.
	\end{equation}
		
	We begin proving the upper semi-continuity of the limit of the integrals. 
	Observe that this result is not a direct consequence of the weak star convergence of the measure sequence since the potential $\F=\{f_n\}_{n=1}^{\infty}$ is not bounded bellow.
	
	\begin{lemma}\label{lema_lim_ut}
		The map $\F_*(\cdot):\mathcal{M}\to\R$ is upper semicontinuous. In particular, we have
		\begin{equation*}
			\limsup_{m \to \infty} \F_*(\mu_{t,k_m})\leq \F_*(\mu_t),\quad\text{for each }t>1.
		\end{equation*} 
	\end{lemma} 
	\begin{proof}		
Let $(\nu_n)_{n\geq1}$ be a sequence of measures in $\mathcal{M}$ which converge to the measure $\nu\in\mathcal{M}$ in the weak-star topology. 
		
		\noindent{\bf Claim:} For any $m\in\N$, we have
		\begin{equation*}
			\limsup_{n\to\infty}\F_*(\nu_n)\leq\int\frac{f_m+C_{aa}(\F)}{m}\,d\nu.
		\end{equation*}
		\noindent $\lhd$	
		For any $K\in\N$, write
		\begin{equation*}
			\Sigma(K):=\Big\{x=x_1x_2\cdots\in\Sigma:\,x_1\cdots x_m\in\{1,\cdots,K\}^m\Big\}.
		\end{equation*}
		Then $\Sigma(K)\uparrow\Sigma$ as $K\uparrow\infty$. Define
		$$\varphi_K:=f_{m}\cdot\chi_{\Sigma(K)}.$$
		
		Since $\varphi_K\leq0$, we have $\varphi_K\downarrow f_{m}$, as $K\to\infty$. Notice that $\Sigma(K)$ is both open and closed,
		so $\chi_{\Sigma(K)}$ is continuous, and hence $\varphi_K$ is bounded and continuous on $\Sigma$. So we have
		\begin{align*}
			\limsup_{n\to\infty}\F_*(\nu_{n})=	&\limsup_{n\to\infty}\inf_m\int \frac{f_{m}+C_{aa}(\F)}{m}\,d\nu_n\leq\limsup_{n\to\infty}\int \frac{f_{m}+C_{aa}(\F)}{m}\,d\nu_n
			\\
			&\leq\limsup_{n\to\infty}\int \frac{\varphi_{K}+C_{aa}(\F)}{m}\,d\nu_n\leq\int \frac{\varphi_{K}+C_{aa}(\F)}{m}\,d\nu.
		\end{align*}
		Now by the monotone convergence theorem, we have
		\begin{equation*}
			\lim_{K\to \infty}\int \frac{\varphi_{K}+C_{aa}(\F)}{m}\,d\nu=\int \frac{f_{m}+C_{aa}(\F)}{m}\,d\nu.
		\end{equation*}
		So the claim holds. \hfill $\rhd$
		
		The result follows directly from the claim.
	\end{proof}
	\begin{remark}
		(1) This above lemma holds under weaker assumptions that the potential $\F=\{f_n\}_{n=1}^{\infty}$
		is almost-additive and has bounded variation with $\sup f_1<\infty$.
		
		(2) Lemma \ref{lema_lim_ut} was presented in \cite[Lemma 4.1]{IY2014}. Since the potential $\F$ is not bounded bellow, thus \cite[Lemma 4.1]{IY2014} was actually not completely proved there. 
	\end{remark}
	Now we turn our attention to the Kolmogorov-Sinai entropy. Let $\xi =\{[a]:a \in\N\}$ be the natural partition of $\Sigma$, for any $\mu \in\mathcal{M}$, its Kolmogorov-Sinai entropy is defined by
	\begin{equation}\label{def-entropy}
		h(\mu)=\lim_{n\to \infty}\frac{H_{\mu}(\xi^n)}{n}=\inf_n\frac{H_{\mu}(\xi^n)}{n},\quad\text{where $H_{\mu}(\xi^n):=\sum_{w\in\N^n}-\mu([w])\log\mu([w])$}
	\end{equation}
	
	In this way, we have the following upper semicontinuous property of the entropy map:
	\begin{lemma}\label{entropy-continuous-lemma}
		Let $(\mu_{t,k_m})_{m\in\N}$ be a sequence given by \eqref{weak-star}. Then 
		\begin{equation}\label{upper-semicontinuous-compact}
			\limsup_{m\to \infty} h(\mu_{t,k_m})\leq h(\mu_t).
		\end{equation}
If there exists a sequence $(t_k)_{k\in\N}$ with $t_k\to s\in(1,\infty]$ such that $\lim_{k\to\infty}\mu_{t_k}=\mu$, then 
		\begin{equation}\label{upper-semicontinuous}
			\limsup_{k\to \infty} h(\mu_{t_k})\leq h(\mu).
		\end{equation}
	\end{lemma}
	\begin{proof}		
		For any $M\in\N$, write $M^n:=\{w_1\cdots w_n\in\N^n:1\leq w_i\leq M,\,1\leq i\leq n\}$ and define
		\begin{equation*}
		H_{\mu_{t,k_m}}(\xi^n|_M):=\sum_{w\in M^n}\phi(\mu_{t,k_m}([w])),\quad\text{where } \phi(x):=-x\log x.
		\end{equation*}
		Notice that for any fixed $M\in\N$ and $n\in\N$, the above summations are finite summations.
		
		\noindent{\bf Claim:} We have the following uniform estimate:
		\begin{equation*}
			0\leq H_{\mu_{t,k_m}}(\xi^n)-H_{\mu_{t,k_m}}(\xi^n|_M)\leq -nC_{t,n}\sum_{i>M}e^{t\sup f_1|_{[i]}}\log (C_{t,n}e^{t\sup f_1|_{[i]}}),
		\end{equation*}
where $C_{t,n}:=\exp(tC_{bv}(\F)+ntC_{aa}(\F)+t(n-1)\sup f_1-nP(t)).$
			
		\noindent $\lhd$	
		Fix $t>1$ and $n\in\N$. By equations \eqref{gibbs}, \eqref{C2} and $\F$ is almost-additive, for all $w\in\N^n$, 
		\begin{align}\label{important-estimate}
			\mu_{t,k_m}([w])
			&\leq\exp(tC_{bv}(\F)+ntC_{aa}(\F)+t(n-1)\sup f_1+t\sup f_1|_{[j(w)]}-nP(t))\notag\\
			&= C_{t,n}\exp(t\sup f_1|_{[j(w)]}),
		\end{align}
		where $j(w)=\max\{w_i:\,1\leq i\leq n\}$. It is seen that $\phi$ is continuous and increasing on $[0,1/e)$, if $M$ is large such that for all $w\in\N^n\backslash M^n$, $C_{t,n}\exp(t\sup f_1|_{[j(w)]})<1/e$, then
		\begin{align*}
			H_{\mu_{t,k_m}}(\xi^n)-H_{\mu_{t,k_m}}(\xi^n|_M)&=\sum_{w\in\N^n\backslash M^n}\phi(\mu_{t,k_m}([w]))\\
			&\leq-\sum_{w\in\N^n\backslash M^n}C_{t,n}e^{t\sup f_1|_{[j(w)]}}\log (C_{t,n}e^{t\sup f_1|_{[j(w)]}})\\
			&\leq-nC_{t,n}\sum_{i>M}e^{t\sup f_1|_{[i]}}\log (C_{t,n}e^{t\sup f_1|_{[i]}}).
		\end{align*}
		Now the claim holds. \hfill $\rhd$
		
		Now by the claim, for any $n$ and $M\in\N$, we have
		\begin{align*}
			&h(\mu_{t,k_m})=\inf_{n\in\N}\frac{H_{\mu_{t,k_m}}(\xi^n)}{n}\leq \frac{H_{\mu_{t,k_m}}(\xi^n)}{n}\\
			\leq&\frac{H_{\mu_{t,k_m}}(\xi^n)-H_{\mu_{t,k_m}}(\xi^n|_M)}{n}+\frac{H_{\mu_{t,k_m}}(\xi^n|_M)-H_{\mu_{t}}(\xi^n|_M)}{n}+\frac{H_{\mu_{t}}(\xi^n|_M)-H_{\mu_{t}}(\xi^n)}{n}+\frac{H_{\mu_{t}}(\xi^n)}{n}\\
			\leq&-nC_{t,n}\sum_{i>M}e^{t\sup f_1|_{[i]}}\log (C_{t,n}e^{t\sup f_1|_{[i]}})+\frac{H_{\mu_{t,k_m}}(\xi^n|_M)-H_{\mu_{t}}(\xi^n|_M)}{n}+\frac{H_{\mu_{t}}(\xi^n)}{n}.
		\end{align*}
		Fix any $\varepsilon>0$ and $n\in\N$. Since $\F$ is summable, by \eqref{eqsum} and \eqref{eqsumt}, $\sum_{i}e^{t\sup f_1|_{[i]}}<\infty$ and $\sum_{i}e^{t\sup f_1|_{[i]}}(t\sup f_1|_{[i]})<\infty$, hence we can choose $M$ large enough such that \begin{equation*}
			C_{t,n}\exp(t\sup f_1|_{[M]})\leq1/e\quad\text{and}\quad-nC_{t,n}\sum_{i>M}e^{t\sup f_1|_{[i]}}\log (C_{t,n}e^{t\sup f_1|_{[i]}})<\varepsilon.
		\end{equation*}
		Since $[w]$ is open and closed for any $w\in\N^n$, we have $\mu_{t,k_m}([w])\to\mu_{t}([w])$ as $m\to\infty$. So
		\begin{equation*}
			\limsup_{m\to\infty}h(\mu_{t,k_m})\leq\varepsilon+\limsup_{m\to \infty}\frac{H_{\mu_{t,k_m}}(\xi^n|_M)-H_{\mu_{t}}(\xi^n|_M)}{n}+\frac{H_{\mu_t}(\xi^n)}{n}\leq\varepsilon+\frac{H_{\mu_t}(\xi^n)}{n}.
		\end{equation*}
		Hence we have
		\begin{equation*}
			\limsup_{m\to\infty}h(\mu_{t,k_m})\leq\varepsilon+\inf_n\frac{H_{\mu_t}(\xi^n)}{n}=\varepsilon+h(\mu_t).
		\end{equation*}
		Consequently, we see that \eqref{upper-semicontinuous-compact} holds.
		
		If there exists a sequence $(t_k)_{k\in\N}$ with $t_k\to s\in(1,\infty]$ such that $\lim_{k\to\infty}\mu_{t_k}=\mu$. Notice that from equations \eqref{weak-star} and \eqref{important-estimate}, we can take the limit as $m\to\infty$ and at both sides, then we have for large $k$ and $w\in\N^n$,
		\begin{equation*}
			\mu_{t_k}([w])\leq \begin{cases}
				C_{s,n}\exp(s\cdot\sup f_1|_{[j(w)]}),\quad &s\in(1,\infty),\\
				\exp(-[j(w)]),\quad &s=\infty.
			\end{cases}
		\end{equation*}
		So, following the similar proof of \eqref{upper-semicontinuous-compact}, we can obtain that \eqref{upper-semicontinuous} holds.
	\end{proof}
	\begin{proof}[Proof of Theorem \ref{equilibrium state}]
		Let $(k_m)_{m \in \mathbb{N}}$ be a sequence such that equation \eqref{weak-star} holds. Since $(P_{k_m}(t))_{m \in \mathbb{N}}$ forms an increasing sequence, then by Lemma \ref{approximation-lemma} and Lemma \ref{lema_lim_ut}, 
		\begin{equation*}
			\lim_{m \to \infty} P_{k_m}(t) = P(t) \quad \text{and} \quad \limsup_{m \to \infty} \mathcal{F}_*(\mu_{t,k_m}) \leq \mathcal{F}_*(\mu_t).
		\end{equation*}
Combining with Lemma \ref{entropy-continuous-lemma}, we deduce that
		\begin{equation*}
			P(t) = \lim_{m\to \infty} P_{k_m}(t) \leq \limsup_{m \to \infty} h(\mu_{t,k_m}) + \limsup_{m \to \infty} \mathcal{F}_*(\mu_{t,k_m}) \leq h(\mu_t) + \mathcal{F}_*(\mu_t).
		\end{equation*}
		This implies that $\mu_t$ is an equilibrium state for $t\mathcal{F}$ when $t > 1$. 
		
Now we prove the uniqueness of the equilibrium state for each $t>1$. From \eqref{important-estimate}, we obtain (here, $C_t:=\exp\big(tC_{bv}(\F)+tC_{aa}(\F)- P_{1}(t)\big)$)
		$$
		\mu_t[i] \leq C_{t}\exp(\sup(tf_1|_{[i]})).
		$$
		Moreover, since $\F$ is almost-additive, then $-\F$ is almost-additive and for all $\mu\in\M$, 
		\begin{equation*}
			-\F_*(\mu)=\lim\limits_{n\to\infty}\int\frac{-f_n+C_{aa}(\F)}{n}\,d\mu=\inf_{n}\int\frac{-f_n+C_{aa}(\F)}{n}\,d\mu\leq\int(-f_1+C_{aa}(\F))\,d\mu.
		\end{equation*}
		Together with \eqref{eqsum} and \eqref{eqsumt}, for any $t > 1$, we have
		\begin{align*}\label{eq1thm1}
			-t\F_*(\mu_t) 
			&\leq t\int(-f_1+C_{aa}(\F))\,d\mu_t\leq \sum_{i \in \mathbb{N}}t\Big(\sup(-f_1|_{[i]})+C_{aa}(\F)\Big) \mu_t[i]\notag\\
			&\leq tC_tC_{aa}(\F)\sum_{i \in \mathbb{N}}\exp(\sup(tf_1|_{[i]}))+C_t \sum_{i \in \mathbb{N}} \sup(-t f_1|_{[i]}) \exp(\sup(t f_1|_{[i]}))
			< \infty \,. 
		\end{align*}
Since $\sup f_1$ is bounded above, then $\F_*(\mu_t)<\infty$. Therefore, by Proposition \ref{IY-variational principle}, we see that $h(\mu_t), |\mathcal{F}_*(\mu_t)|< \infty$, the uniqueness argument in \cite[Theorem 3.5]{MU2001} implies $\mu_t\equiv\mu_{t\mathcal{F}}$. 
		
		To establish tightness of the family $(\mu_t)_{t>1}$, observe that equation \eqref{mu-upper} combined with $\mu_t = \mu_{t\mathcal{F}}$ permits passage to the limit $k\to\infty$ (potentially by passing to a convergent subsequence of $\mu_{t,k}$), yielding
		\begin{equation*}
			\mu_t([i]) \leq \exp\left(C_{bv}(\mathcal{F}) + \sup f_1|_{[i]} - S\right).
		\end{equation*}
		As the right-hand side is $t$-independent, the tightness criterion in Lemma \ref{lematight} applies. Consequently, the family $(\mu_t)_{t>1}$ admits an accumulation point $\mu_{\infty}$ as $t\to\infty$.
	\end{proof} 
	
		
	\section{Proof of Theorem \ref{relatived-pressure-diff} and its consequences}\label{Sec-proof-Theorem-234}
	In this section, we establish the proof of Theorem \ref{relatived-pressure-diff} and analyze the zero temperature limit and entropy continuity at infinity. As direct consequences of Theorem \ref{relatived-pressure-diff}, we further derive two critical results: Theorem \ref{entropy-continuous}  and Theorem \ref{zero-continuous} . 
	
	The proof strategy begins with the following foundational lemma:
	\begin{lemma}\label{upper-Lyap}
		The maps $H,\, L:(1,\infty)\to\R$ are continuous. Moreover, for any $s\in(1,\infty)$, we have $\lim_{t\to s}\mu_{t\F}=\mu_{s\F}$ .
		
	\end{lemma}
	\begin{proof}
		At first we have the following:
		
		\smallskip
		\noindent{\bf Claim}: If $s,\,t_n>1$ are such that $\mu_{t_n\F}\to\mu$ as $t_n\to s$, then $\mu=\mu_{s\F}$. Consequently, the weak limit $\lim_{t\to s}\mu_{t\F}$ exists and is $\mu_{s\F}.$
		
		\noindent $\lhd$	
			Fix $s>1$. Assume that $t_n\in(1,\infty)$ is such that $\mu_{t_n\F}\to\mu$ as $t_n\to s$, then by Lemmas \ref{lema_lim_ut} and \ref{entropy-continuous-lemma}, we have
			$$\lim_{n\to \infty}(h(\mu_{t_n\F})+s\F_*(\mu_{t_n\F}))\leq h(\mu)+s\F_*(\mu).$$
			Combine with Proposition \ref{IY-variational principle} and $P$ is convex on $(1,\infty)$,
			\begin{align*}\label{imp-relativized-1}
				h(\mu)+s \F_*(\mu)&\geq\limsup_{n\to\infty}(H(t_n)+s L(t_n))\\
				&=\limsup_{n\to\infty}(P(t_n)+(s-t_n)L(t_n))\geq P(s)-\limsup_{n\to\infty}|s-t_n|\cdot|L(t_n)|.
			\end{align*}
			Observe that $H(t_n),\,|L(t_n)|<\infty$ for any fixed $t_n>1$, then it follows that
			\begin{equation*}
				h(\mu)+s \F_*(\mu)\geq P(s).
			\end{equation*}
			Therefore, $\mu$ is an equilibrium state for $s\F$, and $\mu\equiv\mu_{s\F}$ by Theorem \ref{equilibrium state}. 

Note that $(\mu_{t\F})_{t>1}$ is tight, and hence, the weak limit $\lim_{t\to s}\mu_{t\F}$ exists and is $\mu_{s\F}$. \hfill $\rhd$
		
		By claim, Lemmas \ref{lema_lim_ut} and \ref{entropy-continuous-lemma}, we have 
		\begin{equation}\label{equilibrium-1}
\limsup_{t\to s}L(t)\leq L(s)\quad\text{and}\quad\limsup_{t\to s}H(t)\leq H(s).
		\end{equation}
		On the other hand, for any $t>1$, by Proposition \ref{IY-variational principle}, we have
		\begin{equation}\label{equilibrium-2}
			P(t)=H(t)+tL(t).
		\end{equation}
		Since $P$ is convex on $(1,\infty)$, and by the first inequality of \eqref{equilibrium-1}, we have
		\begin{equation*}
			\liminf_{t\to s}H(t)=P(s)-s\limsup_{t\to s}L(t)\geq P(s)-sL(s)=H(s).
		\end{equation*}
		Combine with the second inequality of \eqref{equilibrium-1}, then $H$ is continuous on $(1,\infty)$. By using \eqref{equilibrium-2}, we conclude that $L$ is continuous on $(1, \infty)$.
	\end{proof}
	\begin{remark}
		Inspired by Lemma \ref{upper-Lyap}, we conjecture that the weak limit $\lim_{t\to\infty}\mu_{t\F}$ exists.
	\end{remark}
	\begin{proof}[Proof of Theorem \ref{relatived-pressure-diff}]
		By Theorem \ref{equilibrium state}, let $\mu_{(s+t)\F},\mu_{s\F}\in\mathcal{M}$ be equilibrium states respectively of $(s+t)\F$ and $s\F$ (here, we assume that $s+t>1,s>1$). 
		
		(1) By Proposition \ref{IY-variational principle}, we have that
		\begin{align*}
			\left\{
			\begin{aligned}		P(s+t)-P(s)&\geq H(s)+(s+t)L(s)-P(s)=tL(s),\\
				P(s+t)-P(s)&\leq P(s+t)-H(s+t)-sL(s+t)=tL(s+t).
			\end{aligned}
			\right.
		\end{align*}
		This yields the inequalities
		\begin{equation}\label{relatived-diff}
			\left\{
			\begin{aligned}
				L(s)\leq\frac{P(s+t)-P(s)}{t}\leq L(s+t),\ \ \ \text{if}\ t>0,\\
				L(s)\geq\frac{P(s+t)-P(s)}{t}\geq L(s+t),\ \ \ \text{if}\ t<0.
			\end{aligned}
			\right.
		\end{equation}
		By Lemma \ref{upper-Lyap}, $L$ is continuous on $(1,\infty)$, hence it follows from \eqref{relatived-diff} that
		\begin{equation*}
			P'(s)=\lim_{t\to0}\frac{P(s+t)-P(s)}{t}=L(s).
		\end{equation*}
		
		Now we show that $P$ is strictly convex on $(1,\infty)$. If otherwise. Since the function $P$ is $C^1$ and convex on $(1,\infty)$, then there exist $1< t_1<t_2$ such that 
		\begin{align}\label{p} 
			L(t_1)=L(t_2)=L(t),\ \forall\ t\in(t_1,t_2).
		\end{align}
		By Proposition \ref{IY-variational principle}, we have
		\begin{align*}
			P(t_1)=H(t_1)+t_1L(t_1);\quad
			P(t_2)=H(t_2)+t_2L(t_2)=H(t_2)+t_2L(t_1).
		\end{align*}
		Noting that $P'(s)=L(s)$ for all $s\in(1,\infty)$, there exists $\xi\in(t_1,t_2)$ such that
		\begin{equation}\label{p-3}
			\frac{H(t_1)-H(t_2)}{t_1-t_2}+L(t_1)=\frac{P(t_1)-P(t_2)}{t_1-t_2}=P'(\xi)=L(\xi).
		\end{equation}
		According to \eqref{p} and \eqref{p-3}, we conclude that $H(t_1)=H(t_2)$, this means that $$P(t_2)=H(t_1)+t_2L(t_1).$$ Hence $\mu_{t_1}$ is the equilibrium state for $t_2\F$, it contradicts with Theorem \ref{equilibrium state}.
		
		(2) By Lemma \ref{upper-Lyap}, the maps $-L,\,H$ are continuous on $(1,\infty)$, and the weak limit $\lim_{t\to s}\mu_{t\F}=\mu_{s\F}$ for any $s\in(1,\infty)$.
		
		Noting that $P'=L$ and $P$ is strictly convex, then $-L$ is strictly decreasing on $(1,\infty)$. Assume $1<s<t$, then there exists $\xi\in(s,t)$ such that
		\begin{align*}
			H(s)-H(t)&=P(s)-sL(s)-(P(t)-tL(t))=P'(\xi)(s-t)-sL(s)+tL(t)\\
			&=L(\xi)(s-t)-sL(s)+tL(t)
			=s(L(\xi)-L(s))+t(L(t)-L(\xi))>0.
		\end{align*}
		So $H$ is strictly decreasing on $(1,\infty)$. 
	\end{proof}
	\begin{remark}
		The uniqueness of equilibrium states in the proof of Theorem \ref{relatived-pressure-diff} is the key point where the
		Theorem \ref{equilibrium state} is needed. It is interesting to know if it is possible to use the uniqueness of equilibrium states to obtain that the pressure $P$ is $C^1$ on $(1,\infty)$.
	\end{remark}
	As two consequences of Theorem \ref{relatived-pressure-diff}, we now proceed to prove Theorems \ref{entropy-continuous} and \ref{zero-continuous}.
	\begin{proof}[Proof of Theorem \ref{entropy-continuous}]
		By Theorem \ref{equilibrium state},	let $\mu_{\infty}\in\M$ be an accumulation point for the equilibrium state family $(\mu_{t\F})_{t>1}$. Lemma \ref{lema_lim_ut} means that the map $ \F_*:\mathcal{M}\to\R$ is upper semicontinuous, then 
		\begin{equation}\label{L-continuous-infty}
			\liminf_{t \to \infty}L(t)=\liminf\limits_{t\to\infty}\F_*(\mu_{t\F})\leq\limsup\limits_{t\to\infty}\F_*(\mu_{t\F})\leq \F_*(\mu_{\infty}).
		\end{equation}
		
		By Proposition \ref{IY-variational principle} and Theorem \ref{relatived-pressure-diff} (1), for any $\mu\in\M$, we have
		\begin{equation*}
			\F_*(\mu)\leq\limsup\limits_{t\to\infty}\left(\frac{P(t)}{t}-\frac{h(\mu)}{t}\right)=\limsup_{t \to \infty}P'(t)=\limsup_{t \to \infty}L(t).
		\end{equation*}
		This implies that
		\begin{equation}\label{maximizing}
			\F_*(\mu_{\infty})\leq\sup\{\F_*(\mu):\mu\in\mathcal{M}\}\leq\limsup\limits_{t\to\infty}L(t).
		\end{equation}
		Now Theorem \ref{entropy-continuous} just follows from \eqref{L-continuous-infty} and \eqref{maximizing}.
	\end{proof}
	\begin{proof}[Proof of Theorem \ref{zero-continuous}]
		By Theorem \ref{relatived-pressure-diff} (2), the map $H$ is strictly decreasing, this combines with Lemma \ref{entropy-continuous-lemma}, we have
		\begin{equation}\label{zero-1}
			\liminf_{t\to\infty}H(t)=\liminf_{t\to\infty}h(\mu_t)\leq\limsup_{t\to\infty}h(\mu_t)\leq h(\mu_{\infty}).
		\end{equation}
		
		By Proposition \ref{IY-variational principle} and Theorem \ref{entropy-continuous} ($\mu_{\infty}$ is a maximising measure for $\F$), we have,
		\begin{equation*}
			h(\mu_{\infty})\leq P(t)-t \F_*(\mu_{\infty})\leq P(t)-tL(t)=H(t),\ \ \text{for all}\ t>1. 
		\end{equation*}
		Letting $t\to\infty$ and taking the supremum for $\mu\in\M_{\max}(\F)$ on both sides, we have 
		\begin{equation}\label{zero-2}
			h(\mu_{\infty})\leq\sup\{h(\mu):\mu\in\M_{\max}(\F)\}\leq \limsup_{t \to \infty}H(t). 
		\end{equation}
		Then Theorem \ref{zero-continuous} follows directly from \eqref{zero-1} and \eqref{zero-2}.
	\end{proof}
	
	\noindent{\bf Acknowledgement}.
	The author J. Cao would like to thank Yanhui Qu for his guidance and support. J. Cao was partially supported by Nankai Zhide Foundation.
	\appendix 
	\section{The construction of the Gibbs state in compact setting}\label{appendix}
	
	In this Appendix, we recall the construction of the Gibbs state in compact setting, see \cite{Ba2006,Mu2006} for more details. Here, we assume that $(\Sigma,\sigma)$ is a compact topologically mixing shift, and the potential $\F=\{f_n\}_{n=1}^{\infty}$ is almost-additive and has bounded variation. 
	
	At first, define the {\it topological pressure} of $\F$ on $\Sigma$ as 
	\begin{equation*}
		P_{top}(\F):=\lim_{n\to\infty}\frac{1}{n}\log\sum_{|w|=n}\exp(\sup_{x\in[w]}(f_n(x))).
	\end{equation*}
Moreover, we have the following:
	\begin{equation}\label{pressure}
		P_{top}(\F)=\inf_{n}\frac{1}{n}\log\sum_{|w|=n}\exp(\sup_{x\in[w]}(f_n(x))).
	\end{equation}
	According to \cite[Theorem 1 and Theorem 2]{Ba2006}, we see that $P_{G}(\F)=P_{top}(\F)$. So, in compact setting we have that the Gurevic pressure is the same as the topological pressure, and the pressure can be read in any way.
	
	Now	let us recall that the construction of the Gibbs state $\mu$ on compact $\Sigma$ (see \cite[p. 451]{Mu2006}, see also \cite{Ba2006}). At first, for any admissible word $w$ such that $|w|=n$, define
	\begin{equation*}
		\nu_n([w]):=\frac{\sup_{x\in[w]}\exp( f_n(x))}{\sum_{[w]=n}\sup_{x\in[w]}\exp( f_n(x))},
	\end{equation*}
	and hence we have (use \eqref{pressure} and the fact: $P_{G}(\F)=P_{top}(\F)$)
	\begin{align*}
		\frac{\nu_n([w])}{\exp( f_n(x)-P_{G}(\F))}&=\frac{\sup_{x\in[w]}\exp( f_n(x))}{\exp( f_n(x))}\cdot\frac{\exp(nP_{top}(\F))}{\sum_{[w]=n}\sup_{x\in[w]}\exp( f_n(x))}\notag\\
		&\leq \frac{\sup_{x\in[w]}\exp( f_n(x))}{\exp( f_n(x))}\leq \exp(C_{bv}(\F)),\ \ \forall x\in[w].
	\end{align*}
	There exists a subsequence $\{\nu_{n_k}\}$ that converges in the weak-star topology to a
	probability measure $\nu$. At last, the Gibbs state $\mu$ is a limit point of a subsequence of
	\begin{equation*}
		\left\{\frac{1}{m}\left(\nu+\nu\circ\sigma^{-1}+\cdots+\nu\circ\sigma^{-(m-1)}\right)\right\}
	\end{equation*}
	in the weak-star topology, see also \cite[Lemma 16 and Lemma 17]{Mu2006}. Now we conclude that 
	\begin{equation*}
		\frac{\mu[x_1 \ldots x_{n}]}{\exp( f_n(x) - nP_{G}(\F))}\leq\exp(C_{bv}(\F)),\ \ \text{for all}\ x\in\Sigma,\, n\in\N.
	\end{equation*}
	
	\bibliographystyle{siam}
	\bibliography{refs.bib}
	
\end{document}